\documentclass[a4paper,11pt, leqno]{article}

\usepackage{latexsym,enumerate,graphicx}
\usepackage{amsmath,amsthm,amsopn,amstext,amscd,amsfonts,amssymb}
\usepackage[T1]{fontenc}
\usepackage{fullpage,yhmath}
\usepackage{hyperref}
\usepackage{eucal}
\usepackage{color}
\usepackage{verbatim}
\usepackage[normalem]{ulem}
\usepackage{frcursive}

\newenvironment{frcseries}{\fontfamily{frc}\selectfont}{}

\numberwithin{equation}{section}

\newtheorem{theorem}{Theorem}

\newtheorem{proposition}{Proposition}
\newtheorem{remark}{Remark}

\numberwithin{theorem}{section}
\numberwithin{corollary}{section}
\numberwithin{lemma}{section}
\numberwithin{definition}{section}
\numberwithin{proposition}{section}
\numberwithin{remark}{section}

\newcommand{\RR}{\mathbb R^N}

\newcommand{\R}{\mathbb R}

\newcommand{\medint}{-\kern  -,375cm\int}
\newcommand{\dint}{\displaystyle\int}

\newcommand{\di}{\mathop{\mathrm{d}\!}}

\newcommand{\dive}{\mathop{\mathrm{div}}}

\makeatletter
\@namedef{subjclassname@2020}{  \textup{2020} Mathematics Subject Classification}
\makeatother
\date{}

\begin{document}
\title{Kohler-Jobin inequality for $p$-Laplace operator}

\author{Francesco Chiacchio$^1$, Vincenzo Ferone$^{1,3}$, Anna Mercaldo$^1$, Jing Wang$^2$}
\footnotetext[1]
{{Dipartimento di Matematica e Applicazioni ``Renato Caccioppoli'', Universit\`a degli Studi di Napoli Federico II, Via Cintia, Complesso Universitario Monte S. Angelo, 80143 Napoli, Italy.}}

%
\footnotetext[2]
{{Department of Mathematics, Purdue University,
West Lafayette, Indiana, 47907, USA.}}
\footnotetext[3]
{{Corresponding author: ferone@unina.it}}
\maketitle

\setcounter{tocdepth}{1}

\begin{abstract}
A sharp lower bound for the first Dirichlet eigenvalue of the $p$-laplacian is derived for sets with prescribed $p$-torsional rigidity. The result provides an extension of the classical spectral inequality due to Kohler-Jobin. The proof is based on a careful analysis of the generalized $p$-torsional rigidity and on a sharp mass comparison result.

\vskip 0.5cm
\noindent\textit{Keywords:} Symmetrization, Kohler-Jobin inequality, reverse H\"older inequality.

\noindent\textit{MSC 2020: \textup{35P15, 47J10, 35J92.}}
\end{abstract}


\section{Introduction}\label{intr}

Given an open set $\Omega\subset\R^N$ with finite measure, for $p>1$ we consider the following
eigenvalue problem
\begin{equation} \label{eq.0}
\left\{
\begin{array}
[c]{lll}%
 -\Delta_pu=\lambda |u|^{p-2}u & & \text{in }%
\Omega,\\
\\
u=0 & & \text{on }\partial\Omega,
\end{array}
\right. %
\end{equation}
where, as usual $\Delta_pu=\dive(|Du|^{p-2}Du)$ denotes the $p$-Laplacian.
It is well known that the first eigenvalue $\lambda_1(\Omega)$ has the variational characterization
\begin{equation}\label{eigen}
\lambda_1(\Omega)=\min_{w\in W_{0}^{1,p} ( \Omega)\backslash\{0\}}
\dfrac{\displaystyle\int_\Omega|Dw|^p\, \di x}{\displaystyle\int_\Omega |w|^p\, \di x}.
\end{equation}
Moreover, there exists a positive eigenfunction $u_1$ associated with $\lambda_1(\Omega)$ that attains the minimum in \eqref{eigen}.

A second quantity naturally associated with $\Omega$ is the
  $p$-torsional
    rigidity
\begin{equation}\label{TT}
T(\Omega)=\max_{w\in W_{0}^{1,p} ( \Omega )\backslash\{0\}}\dfrac{\displaystyle\left(\int_\Omega w\, \di x\right)^p}{\displaystyle\int_\Omega|Dw|^p\, \di x}.
\end{equation}
It is known that its maximum is attained at $w=v$, where $v$ is the torsion function that solves the following boundary value problem
\begin{equation} \label{ve}
\left\{
\begin{array}
[c]{lll}%
-\Delta_p v=1& & \text{in }%
\Omega,\\
\\
v=0 & & \text{on }\partial\Omega,
\end{array}
\right.
\end{equation}
and it follows that 
\begin{equation}\label{TQ}
T(\Omega)=\left(\int_\Omega v(x)\, \di x\right)^{p-1}.
\end{equation}

The extremal sets for the principal frequency and the torsional rigidity have been extensively studied in the literature. 
When $p=2$, these quantities are closely related to classical isoperimetric-type inequalities.
Among sets of given measure, the ball minimizes $\lambda_1(\Omega)$, as stated in the Lord Rayleigh conjecture and proven by Faber and Krahn (\cite{Fa,Kr}), and maximizes $T(\Omega)$, as stated in the Saint-Venant conjecture and proven by P\'olya (\cite{Po}). 
Moreover, in \cite{PS} P\'olya and Szeg\H o stated the stronger conjecture that among sets with fixed torsional rigidity, the ball minimizes the principal frequency. This conjecture was firstly proved by Kohler-Jobin in \cite{KMCb,KMZa} using a new rearrangement technique known as \emph{transplantation \`a integrales de Dirichlet \'egales}. 
For $p>1$, a nonlinear version of Kohler-Jobin inequality was established in \cite{Br} using a similar approach. More precisely, the following isoperimetric inequality has been proved:
\begin{equation}\label{ineqKJ}
T(\Omega)^{\frac p{Np+p-N}} \lambda_1(\Omega) \ge T(B)^{\frac p{Np+p-N}} \lambda_1(B),
\end{equation}
where $B$ is any ball.

Let us observe that inequality \eqref{ineqKJ} can be equivalently stated by saying that for any fixed $\Omega\subset\R^N$, if $B$ is the ball such that $T(B)=T(\Omega)$, then
\begin{equation}\label{ineKJ}
\lambda_1(\Omega)\ge\lambda_1(B).
\end{equation}
However, the proof in \cite{KMCb}, \cite{KMZa}, \cite{Br} relies on the quoted method introduced in \cite{KMCb}. In the present paper we intend to show that \eqref{ineqKJ} can be proved for a more general notion of torsional rigidity making use of standard rearrangement arguments. Indeed, the torsional rigidity $T(\Omega)$ can be seen as a particular case of a ``generalized torsional rigidity'', firstly introduced in \cite{Ba} in the case  $p=2$, defined, for $\alpha \in \R$, as
\begin{equation}
    \label{torgen_max_intro}
Q_p(\alpha, \Omega)=
\sup_{w\in 
W_{0}^{1,p} ( \Omega)}
\left\{-\int_\Omega|Dw(x)|^p\, \di x+\alpha\int_\Omega |w(x)|^p\, \di x+p\int_\Omega w(x)\, \di x\right\}.
\end{equation}
For any $\alpha \in (-\infty, \lambda_1(\Omega))$,
the maximum in \eqref{torgen_max_intro} is attained at the generalized torsion function $v$, which solves the problem
\begin{equation}
    \label{ve_intro}
\left\{
\begin{array}
[c]{lll}%
-\Delta_pv=\alpha v^{p-1}+1& & \text{in }%
\Omega,\\
\\
v=0 & & \text{on }\partial\Omega.
\end{array}
\right. 
\end{equation}
In particular, when $\alpha=0$, the generalized torsional rigidity reduces to the $p$-torsional rigidity, that is,
\begin{equation}\label{TQa}
T(\Omega)=\left(\dfrac{Q_p(0,\Omega)}{p-1}\right)^{p-1}.
\end{equation}

In this paper we will prove that, for any $\alpha \in (-\infty, \lambda_1(\Omega))$ and for any  set $\Omega\subset\RR$ with finite measure, it holds
\begin{equation}\label{tor-alpha}
\lambda_1(\Omega) \geq \lambda_1(B_{\alpha})\quad\text{where}\ \ B_{\alpha}\ \ \text{is a ball s.t.} \ Q_p(\alpha, B_{\alpha})=Q_p(\alpha, \Omega).
\end{equation}

Clearly, when $\alpha = 0$, the above statement implies \eqref{ineKJ}, and therefore recovers the Kohler--Jobin inequality. Our result can be viewed as a one-parameter extension of this inequality, in which the usual torsional rigidity is replaced by the generalized torsional rigidity $Q_p(\alpha,\Omega)$.

The proof is based on the monotonicity, with respect to $\alpha$, of the radius $r(\alpha)$ of the ball determined by
\[
Q_p(\alpha,\Omega)=Q_p(\alpha,B_{r(\alpha)}).
\]
More precisely, we show that the map $\alpha\mapsto r(\alpha)$ is decreasing. This monotonicity follows from a sharp mass comparison result between the generalized torsion function in $\Omega$ and the corresponding radial solution in the ball $B_{r(\alpha)}$. The analysis also requires a careful study of the dependence of $Q_p(\alpha,\Omega)$ on both the parameter $\alpha$  and the domain $\Omega$, including a differentiation formula 
with respect to $\alpha$.

Let us observe that related ideas appear in the classical papers of Kohler--Jobin \cite{KMCa,KMZa,KMZb}, where the usual Kohler--Jobin inequality is obtained by different techniques. 

In contrast, the present approach relies only on standard rearrangement arguments and on the variational structure of the generalized torsional rigidity. Since the final inequality is obtained through a limiting argument as $\alpha\to\lambda_1(\Omega)^-$, this method does not seem to yield a characterization of the equality case.
The paper is structured as follows. 
In Section \ref{s1}, we introduce some notation and we collect some preliminary results about the first eigenvalue and the torsional rigidity defined in \eqref{eigen} and \eqref{TT}, respectively. In Section \ref{s2} we prove some properties of the generalized torsional rigidity defined in \eqref{torgen_max_intro},
both in a general domain and in balls.
In Section \ref{s3} we prove the comparison result mentioned above, while in Section \ref{s4} we prove the main result.

\section{Preliminary results}\label{s1}
This section is devoted to classical symmetrization tools and to some basic properties of the first eigenvalue and of the $p$-torsional rigidity in the Euclidean setting.

\subsection{Schwarz symmetrization}

Let $u$ be a measurable function defined on a measurable set $\Omega\subset\R^N$ of finite measure. The distribution function of $u$ is the map from $[0,+\infty)$ into $[0,|\Omega|]$ defined by
\[
\mu(t)=|\{x\in\Omega:\ |u(x)|>t\}|.
\]
The function $\mu$ is non-increasing and right-continuous. The decreasing rearrangement of $u$ is defined by
\begin{equation}\label{riord}
u^*(s)=\inf\{t\ge0:\mu(t)\le s\},\qquad 0<s\le |\Omega|.
\end{equation}
Let $\Omega^\sharp$ be the ball centered at
the origin having the same measure as $\Omega$, namely
\[
\Omega^\sharp = B_R := B_R(0), \quad \text{where } R>0 \text{ satisfies } \omega_N R^N = |\Omega|.
\]
Here and in the sequel, $\omega_N$ stands for the measure of the unit ball in $\mathbb{R}^N$.

The Schwarz symmetrization $u^\sharp$ of $u$ is the radial
and radially non-increasing function  defined as follows

\[
u^\sharp(x) = u^*(\omega_N|x|^N), \quad \text{with } 
x \in \Omega^\sharp .
\]
By construction $u$ and $u^\sharp$ are equimeasurable, and therefore, for every $1\le p\le\infty$, it holds that
\begin{equation}\label{inv}
\|u\|_{L^p(\Omega)}=\|u^\sharp\|_{L^p(\Omega^\sharp)} .
\end{equation}
We shall use the classical P\'olya-Szeg\H{o} 
principle (see, e.g., \cite{BZ}) 
\begin{theorem}\label{polya_szego}
Let $p>1$ and let $u\in W_0^{1,p}(\Omega)$. Then $u^\sharp\in W_0^{1,p}(\Omega^\sharp)$ and
\[
\int_\Omega |Du|^p\,\di x\ge \int_{\Omega^\sharp}|Du^\sharp|^p\,\di x.
\]
\end{theorem}
We also recall the Hardy-Littlewood inequality
\begin{equation}\label{Har}
\int_\Omega |f(x)g(x)|\,\di x\le \int_{\Omega^\sharp} f^\sharp(x)g^\sharp(x)\,\di x = \int_0^{|\Omega|}f^*(s)g^*(s)\,\di s
\end{equation}
and the following characterization (cf. \cite{ChongRice}).

\begin{proposition}\label{relazione}
Let $f,g\in L_+^1(\Omega)$. Then the following statements are equivalent:
\begin{equation}
\int_0^t f^*(s)\,\di s\le \int_0^t g^*(s)\,\di s\qquad \forall t\in[0,|\Omega|],
\end{equation}
\begin{equation}
\int_\Omega F(f)\,\di x\le \int_\Omega F(g)\,\di x
\end{equation}
for every convex, nonnegative, Lipschitz continuous function $F$ such that $F(0)=0$.
\end{proposition}

\subsection{First eigenvalue and torsional rigidity}

For $p>1$ and an open set $\Omega\subset\R^N$ of finite measure, we consider the first eigenvalue problem
\begin{equation}
\label{p_Lapl}
\left\{
\begin{array}
[c]{lll}%
 -\dive( |Du|^{p-2}Du)=\lambda  |u|^{p-2}u & & \text{in }%
\Omega,\\
\\
u=0 & & \text{on }\partial\Omega.
\end{array}
\right.
\end{equation}
It is well known that its first eigenvalue, $\lambda_1(\Omega)$, is the minimum of the Rayleigh quotient
\[
\lambda_1(\Omega)=\min_{w\in W_{0}^{1,p}( \Omega)\backslash\{0\}}
\frac{\|Dw\|^p_{L^p(\Omega)}}{\|w\|^p_{L^p(\Omega)}}
\]
and that the above minimization problem is equivalent to the weak form of \eqref{p_Lapl} with $\lambda = \lambda_1(\Omega)$, that is,
\begin{equation*}
\int_\Omega |\nabla u|^{p-2} \nabla u \cdot \nabla \varphi \, dx = \lambda \int_\Omega |u|^{p-2}u \varphi \, dx \quad \forall \varphi \in W^{1,p}_0(\Omega).
\end{equation*}
Furthermore, see e.g. \cite{Lin}, 
$\lambda_1(\Omega)$ is simple and  a corresponding 
eigenfunction has one sign within $\Omega$.
Combining Theorem \ref{polya_szego} with \eqref{inv} yields the Faber-Krahn inequality.

\begin{theorem}\label{KR}
Let $\Omega$ be an open set in $\R^N$ with finite measure. Then
\begin{equation}\label{FK}
\lambda_1(\Omega)\ge \lambda_1(\Omega^\sharp).
\end{equation}
\end{theorem}

Similarly, Theorem \ref{polya_szego} implies the Saint-Venant inequality for the torsional rigidity.

\begin{theorem}\label{SV}
Let $\Omega$ be an open set in $\R^N$ with finite measure. Then
\begin{equation}\label{SVe}
T(\Omega)\le T(\Omega^\sharp).
\end{equation}
\end{theorem}

We add some results concerning the properties of the first eigenvalue on a ball $B_R$.
We firstly observe that it  satisfies the scaling law
\begin{equation}
\label{rescale}
\lambda_1(B_R)=R^{-p}\lambda_1(B_1).
\end{equation}
  Furthermore Theorem \ref{KR} implies that the first eigenfunction is radial, so we have
\begin{equation}
\label{1dim}
\lambda_1(B_R)
=\min_{w\in W^{1,p}_{0}
\left(
(0,R),\rho^{N-1}
\right)\setminus\{0\}}
\frac{\dint_{0}^{R} |w'(\rho)|^{p}\rho^{N-1}\,\di\rho}
{\dint_{0}^{R} |w(\rho)|^{p}\rho^{N-1}\,\di\rho},
\end{equation}
and the first eigenfunction $u_1>0$, which achieves the minimum above, solves the problem
\begin{equation} 
\label{2eq}
\left\{
\begin{array}
[c]{lll}
-\bigl(|u_1'|^{p-2}
u_1'\rho^{N-1}\bigr)'
=
\lambda_1(B_R)u_1^{p-1}\rho^{N-1}& & \rho\in
(0,R),\\
\\
u_1'(0)=0,\qquad u_1(R)=0. & & 
\end{array}
\right.
\end{equation}
 Finally, equality \eqref{rescale}, together with a result established  in \cite{GMS} and \cite{Lam}, immediately implies the following proposition.

\begin{proposition}\label{half}
The following differentiation formula holds
\begin{equation}\label{shape}
\frac{d}{dr}\lambda_1(B_r) 
= -pr^{-p-1}\lambda_1(B_1) 
= -(p-1)N\omega_N r^{N-1} \frac{|u_1'(r)|^p}{\|u_1\|_p^p} .
\end{equation}
\end{proposition}

\section{A generalized torsional rigidity}\label{s2}

Let $\Omega $ be an open  set of $ \R^N$ with finite measure. For $\alpha\in(-\infty,\lambda_1(\Omega))$ we consider the following generalization of torsional rigidity
\begin{equation}\label{QQ}
Q_p(\alpha, \Omega)=\sup_{w\in W_{0}^{1,p} ( \Omega )}\left\{-\int_\Omega|Dw(x)|^p\, \di x+\alpha\int_\Omega |w(x)|^p\, \di x+p\int_\Omega w(x)\, \di x\right\}.
\end{equation}
As we will see,
the maximum of the functional
\begin{equation}
\label{F}
F_\alpha(w):=-\int_{\Omega }|Dw(x)|^{p}\,\di x+\alpha \int_{\Omega
}|w(x)|^{p}\,\di x+p\int_{\Omega }w(x)\di x\,,
\end{equation}
is achieved just
for $w=v$, where $v$ is the 
unique weak solution (see Proposition \ref{prop} to problem
\begin{equation} \label{veq.0}
\left\{
\begin{array}
[c]{lll}%
-\Delta_pv=\alpha  |v|^{p-2}v+1 & & \text{in }%
\Omega,\\
\\
v=0 & & \text{on }\partial\Omega\,.
\end{array}
\right. 
\end{equation}
This means that
\[
 v\in W_{0}^{1,p} ( \Omega ),
 \]
 and
 \begin{equation}\label{vweak}
  \int_{\Omega} |D v|^{p-2}D vD \varphi \,\di x= \alpha \int_{\Omega}|  v|^{p-2}  v \varphi \,\di x +\int_{\Omega}  \varphi\, \di x\,, 
  \notag
 \end{equation}
 for every $\varphi \in W_0^{1,p}(\Omega)$.

\noindent An easy calculation proves that the following equality holds true:
\begin{equation}\label{Q}
Q_p(\alpha, \Omega)=(p-1)\int_\Omega v(x)\, \di x.
\end{equation}
As observed in the introduction, when $\alpha=0$, the generalized torsional rigidity reduces to the $p$-torsional rigidity in the sense that \eqref{TQa}
holds true.

We now derive some useful properties of $Q_p(\alpha, \Omega)$ which are consequences of its  definition.
\begin{proposition}\label{prop}
Let $\Omega\subset \R^N$, with finite measure. Then
\begin{enumerate}[\rm(a)]
\item\label{(a)} $\displaystyle Q_p(\alpha, \Omega)$ is finite$\quad\iff\quad -\infty<\alpha<\lambda_1(\Omega);$

\item\label{(b)}
for any  $\alpha < \lambda _{1} \left( \Omega  \right)$
the functional $F_\alpha(w)$, defined in \eqref{F},
 has a unique maximizer 
 $v\ge0$, which  is the unique weak solution to problem \eqref{veq.0},
and therefore \eqref{Q} holds true; 

\item\label{(c)} $\displaystyle Q_p(\alpha, \Omega)$ is increasing with respect to $\Omega$ (in the sense of inclusion), that is,
\[\Omega_1\subset\Omega_2\quad\Longrightarrow \quad 
Q_p(\alpha,\Omega_1)
\le
Q_p(\alpha,\Omega_2);
\]
\item\label{(d)}  if $\alpha<\lambda_1(\Omega^\sharp)$\,, then 
\[Q_p(\alpha,\Omega)\le Q_p(\alpha,\Omega^\sharp).
\]
\end{enumerate}
\end{proposition}
\begin{proof} We prove in sequence the various items.
\vskip.2cm
\noindent \underline{Item (\ref{(a)})}
\vskip.3cm

{  

Let us suppose that $-\infty<\alpha<\lambda_1(\Omega)$.
 For any  $ w \in W_{0}^{1,p} ( \Omega)$ we have
\begin{equation*}
F_{\alpha}(w)\leq \left( \alpha -\lambda _{1}\left( \Omega\right) \right)\int_{\Omega }|w(x)|^{p}\,\di x+p\int_{\Omega
}\left\vert w(x)\right\vert \di x.
\end{equation*}
For $\varepsilon>0$, Young inequality yields
\begin{equation*}
F_{\alpha}(w)\leq \left( \alpha+\varepsilon -
\lambda _{1} \left( \Omega  \right) \right)
\int_{\Omega }|w(x)|^{p}\,\di x
+C(\varepsilon)|\Omega|,\qquad\forall w \in
W_{0}^{1,p} ( \Omega)
\end{equation*}
and, if $\varepsilon$ is sufficiently small, we have
\begin{equation*}
F_{\alpha}(w)\leq C,\qquad
\forall w \in 
W_{0}^{1,p} ( \Omega),
\end{equation*}
which proves that $Q_p(\alpha,\Omega)$ is finite.

%
We now suppose that $Q_p(\alpha,\Omega)$ is finite. Let us assume, by contradiction, that
\begin{equation}
\text{$\alpha \geq \lambda _{1}$}\left( \text{$\Omega $}\right) .  \label{a>}
\end{equation}
Consider the family of functions $\left\{ tw_{1}\right\} $
where $t>0$ and $w_{1}$ is the positive eigenfunction corresponding to $\lambda _{1}$$\left( \text{$\Omega $}
\right) $, such that
\begin{equation*}
\left\Vert w_{1}\right\Vert _{L^{p}(\Omega)}=1.
\end{equation*}
Using (\ref{a>}), we have
\begin{equation*}
F_{\alpha}\left( tw_{1}\right) =-t^{p}\int_{\Omega }|Dw_{1}(x)|^{p}\,\di x
 +\alpha t^{p}\int_{\Omega }w_{1}(x)^{p}\,\di x+pt\int_{\Omega
}w_{1}(x)\, \di x
\end{equation*}

\begin{equation*}
=pt\int_{\Omega }w_{1}(x)\di x+t^{p}\left( \alpha -\text{$\lambda _{1}
$}\left( \text{$\Omega $}\right) \right)\int_{\Omega }w_{1}(x)^{p}\,\di x\geq pt\int_{\Omega
}w_{1}(x)\di x.
\end{equation*}
It follows that
\begin{equation*}
\lim_{t\rightarrow +\infty }F_{\alpha}\left( tw_{1}\right) =+\infty 
\end{equation*}
and therefore
\begin{equation*}
Q_{p}(\alpha ,\Omega ) =+\infty .
\end{equation*}
This is a contradiction, as we are under the assumption that 
$Q_{p}(\alpha ,\Omega )$ is finite.

\vskip.2cm
\noindent \underline{Item (\ref{(b)})}
\vskip.3cm

We begin by proving that for all $ \alpha < \lambda_{1}\left( \Omega \right)$
the functional $F$ attains a maximum.  We  assume that $\alpha >0$, since, otherwise, 
the proof becomes even simpler.

From the previous considerations, we know that
\begin{align}
\sup_{w\in W_{0}^{1,p} (\Omega)} F_{\alpha}(w) &= \sup_{w\in W_{0}^{1,p} (\Omega )} \left\{ -\int_{\Omega} |Dw(x)|^{p}\,\di x \right. \label{l>} \\
&\quad \left. +\alpha \int_{\Omega }|w(x)|^{p}\,\di x +p\int_{\Omega }w(x)\di x\right\} < + \infty. \nonumber
\end{align}
Let $\{w_{k}\}_{k\in \mathbb{N}}\subset W_{0}^{1,p} ( \Omega)$ be a maximizing sequence. 
By (\ref{l>}), there exists a constant $C$ such that
\begin{equation*}
\int_{\Omega }|Dw_{k}(x)|^{p}\,\di x-\alpha \int_{\Omega}
|w_{k}(x)|^{p}\,\di x-p\int_{\Omega }w_{k}(x)\di x\leq C
\qquad \forall k\in \mathbb{N}.
\end{equation*}
Hence, for every $k\in\mathbb{N}$ we obtain
$$
\int_{\Omega }|Dw_{k}(x)|^{p}\,\di x
\leq \frac{\alpha}{\lambda _{1}\left( \Omega \right)}
       \int_{\Omega }|Dw_{k}(x)|^{p}\,\di x
      +p\int_{\Omega }|w_{k}(x)|\,\di x+C.
$$
 Young inequality ensures that  $\forall \varepsilon >0$ there exists 
$C_{\varepsilon }>0$ such that
\begin{equation*}
\left(1-\frac{\alpha}{\lambda _{1}\left( \Omega \right)}\right)
\int_{\Omega }|Dw_{k}(x)|^{p}\,\di x
\leq \varepsilon \int_{\Omega }|w_{k}(x)|^{p}\,\di x
   +C_{\varepsilon}|\Omega |+C.
\end{equation*}
Finally, using the continuous embedding of 
$W_{0}^{1,p}\left( \Omega  \right)$
into 
$L^{p}\left( \Omega  \right)$, the arbitrariness of 
$\varepsilon$ implies that
\begin{equation}
\int_{\Omega }|Dw_{k}(x)|^{p}\,\di x\leq C
\qquad \forall k\in\mathbb{N},  \label{Equi}
\end{equation}
where, here and in the sequel, $C$ denotes a constant whose value may vary 
from line to line, but which does not depend on the significant parameters 
of the problem.

Finally, from the compact embedding of 
$W_{0}^{1,p}\left( \Omega  \right)$
into
$L^{p}\left( \Omega   \right)$
one deduces that, up to a not relabelled subsequence, 
there exists a function $w\in W_{0}^{1,p}\left( \Omega  \right)$ 
such that
\begin{equation*}
\left\{ 
\begin{array}{cc}
w_{k}\rightarrow w & \text{weakly in }W_{0}^{1,p}\left( \Omega  \right), \\
\\
w_{k}\rightarrow w & \text{strongly in }L^{p}\left( \Omega  \right).
\end{array}
\right.
\end{equation*}
From this, in a standard way, the claim follows immediately. Hence, since $F_{\alpha}$ attains its 
maximum at $w$, its Euler equation in \eqref{veq.0} admits at least one solution.

We now address the issues related to uniqueness. First of all we observe that any solution to problem \eqref{veq.0} is nonnegative. Indeed, using $v_-=\max\{-v,0\}$ as test function in \eqref{veq.0}, we have
\begin{equation*}
\int_{\{v<0\}}|Dv|^p\, \di x-\alpha\int_{\{v<0\}} |v|^p\, \di x=\int_{\{v<0\}} v\, \di x\,.
\end{equation*}
Since the right-hand side is negative and $\alpha<\lambda_1(\Omega)$, this   gives a contradiction if $v_-\not \equiv0$.

The uniqueness of the solution to problem \eqref{veq.0} can be proven using an argument which goes back to \cite{BO} (for the case $p=2$) and which has been used in \cite{DS} (and subsequently refined in \cite{Lin}).

Let $u\not\equiv v$ be two solutions to problem \eqref{veq.0}. For $0<\sigma<k<+\infty$ we consider the test functions
\begin{equation}\label{test}
\varphi_1=u+\sigma-\frac{\bigl( v+\sigma\bigr)^p}{(u+\sigma)^{p-1}},\qquad
\varphi_2=v+\sigma-\frac{\bigl( u+\sigma\bigr)^p}{(v+\sigma)^{p-1}}\,.
\end{equation}
It is immediate to observe that $\varphi_1, \varphi_2\in
W_{0}^{1,p} ( \Omega)$, so we can use $\varphi_1$ in the equation satisfied by $u$ and $\varphi_2$ in the equation satisfied by $v$. Using the notation $u_\sigma=u+\sigma$, $v_\sigma=v+\sigma$, we obtain
\begin{eqnarray*}
\int_\Omega|Du|^p\, \di x-p\int_\Omega|Du|^{p-2}DuD\bigl( v_\sigma\bigr)\frac{\bigl( v_\sigma\bigr)^{p-1}}{(u_\sigma)^{p-1}}\, \di x
+(p-1)\int_\Omega|Du|^{p}\frac{\bigl( v_\sigma\bigr)^{p}}{(u_\sigma)^{p}}\, \di x=\\
=\alpha\int_\Omega \left(\frac u{u_\sigma}\right)^{p-1}\left((u_\sigma)^p-\bigl(v_\sigma\bigr)^p\right)\, \di x+\int_\Omega \left(u_\sigma-\frac{\bigl(v_\sigma\bigr)^p}{(u_\sigma)^{p-1}}\right)\, \di x,
\\
\int_\Omega|Dv|^p\, \di x-p\int_\Omega|Dv|^{p-2}DvD\bigl(u_\sigma\bigr)\frac{\bigl( u_\sigma\bigr)^{p-1}}{(v_\sigma)^{p-1}}\, \di x
+(p-1)\int_\Omega|Dv|^{p}\frac{\bigl( u_\sigma\bigr)^{p}}{(v_\sigma)^{p}}\, \di x=\\
=\alpha\int_\Omega \left(\frac v{v_\sigma}\right)^{p-1}\left((v_\sigma)^p-\bigl(u_\sigma\bigr)^p\right)\, \di x+\int_\Omega \left(v_\sigma-\frac{\bigl(u_\sigma\bigr)^p}{(v_\sigma)^{p-1}}\right)\, \di x.
\end{eqnarray*}
Summing the above equalities and using Young inequalities
\begin{equation*}
\left|p|Du|^{p-2}DuD\bigl(v_\sigma\bigr)\frac{\bigl(v_\sigma\bigr)^{p-1}}{(u_\sigma)^{p-1}}\right|
\le |D\bigl(v_\sigma\bigr)|^p+(p-1)|Du|^{p}\frac{\bigl(v_\sigma\bigr)^{p}}{(u_\sigma)^{p}}
\end{equation*}
\begin{equation*}
\left|p|Dv|^{p-2}DvD\bigl(u_\sigma\bigr)\frac{\bigl(u_\sigma\bigr)^{p-1}}{(v_\sigma)^{p-1}}\right|
\le |D\bigl(u_\sigma\bigr)|^p+(p-1)|Dv|^{p}\frac{\bigl(u_\sigma\bigr)^{p}}{(v_\sigma)^{p}}
\end{equation*}
we obtain
\begin{eqnarray*}
\alpha\left(\int_\Omega \left(\frac u{u_\sigma}\right)^{p-1}\left((u_\sigma)^p-\bigl(v_\sigma\bigr)^p\right)\, \di x
+\int_\Omega \left(\frac v{v_\sigma}\right)^{p-1}\left((v_\sigma)^p-\bigl(u_\sigma\bigr)^p\right)\, \di x\right)+\\
+\int_\Omega \left(u_\sigma-\frac{\bigl(v_\sigma\bigr)^p}{(u_\sigma)^{p-1}}\right)\, \di x
+\int_\Omega \left(v_\sigma-\frac{\bigl(u_\sigma\bigr)^p}{(v_\sigma)^{p-1}}\right)\, \di x\\
\ge\int_\Omega|Du|^p-\int_\Omega|D\bigl(u_\sigma\bigr)|^p\, \di x+\int_\Omega|Dv|^p-\int_\Omega|D\bigl(v_\sigma\bigr)|^p\, \di x\ge0.
\end{eqnarray*}
Therefore we have:
\begin{eqnarray*}
\alpha\left(\int_\Omega \left(\frac u{u_\sigma}\right)^{p-1}\left((u_\sigma)^p-\bigl(v_\sigma\bigr)^p\right)\, \di x
+\int_\Omega \left(\frac v{v_\sigma}\right)^{p-1}\left((v_\sigma)^p-\bigl(u_\sigma\bigr)^p\right)\, \di x\right)+\\
+\int_\Omega \left(u_\sigma-\frac{\bigl(v_\sigma\bigr)^p}{(u_\sigma)^{p-1}}\right)\, \di x
+\int_\Omega \left(v_\sigma-\frac{\bigl(u_\sigma\bigr)^p}{(v_\sigma)^{p-1}}\right)\, \di x\ge0,
\end{eqnarray*}
or equivalently:
\begin{eqnarray}\label{passig}
\alpha\int_\Omega \left(\left(\frac u{u_\sigma}\right)^{p-1}-\left(\frac v{v_\sigma}\right)^{p-1}\right)\left((u_\sigma)^p-(v_\sigma)^p\right)\, \di x\ge\qquad\qquad\\
\notag\ge\int_\Omega \left(\left(\frac 1{v_\sigma}\right)^{p-1}-\left(\frac 1{u_\sigma}\right)^{p-1}\right)\left((u_\sigma)^p-(v_\sigma)^p\right)\, \di x.
\end{eqnarray}
Now we pass to the limit as $\sigma  $ goes to zero in the last inequality. 
Firstly, since, 
\[
\lim_{\sigma\to 0}\left(\left(\frac u{u_\sigma}\right)^{p-1}-\left(\frac v{v_\sigma}\right)^{p-1}\right)=0\,, \qquad \hbox{a. e. in }\Omega\,,
\]
and,
\[
\left |\left(\left(\frac u{u_\sigma}\right)^{p-1}-\left(\frac v{v_\sigma}\right)^{p-1}\right)\left((u_\sigma)^p-(v_\sigma)^p\right)\right |\le 2(u^p+v^p+2^p)
\,,
\]
Lebesgue dominated convergence theorem implies
\begin{equation}\label{leb}
\lim_{\sigma\rightarrow0^+}\int_\Omega \left(\left(\frac u{u_\sigma}\right)^{p-1}-\left(\frac v{v_\sigma}\right)^{p-1}\right)\left((u_\sigma)^p-(v_\sigma)^p\right)\, \di x=0,
\end{equation}
Therefore, for $\sigma \to 0$, inequality \eqref{passig} gives
\begin{equation}\label{abs}
\liminf_{\sigma\rightarrow0^+}\int_\Omega \left(\left(\frac 1{v_\sigma}\right)^{p-1}-\left(\frac 1{u_\sigma}\right)^{p-1}\right)\left((u_\sigma)^p-(v_\sigma)^p\right)\, \di x
\le0,
\end{equation}
Moreover, since 
\begin{eqnarray*}
\left(\left(\frac 1{v_\sigma}\right)^{p-1}-\left(\frac 1{u_\sigma}\right)^{p-1}\right)\left((u_\sigma)^p-(v_\sigma)^p\right)\ge0,
\end{eqnarray*}
if $u\not\equiv v$, by Fatou lemma we have
\begin{eqnarray*}
\liminf_{\sigma\rightarrow0^+}\int_\Omega \left(\left(\frac 1{v_\sigma}\right)^{p-1}-\left(\frac 1{u_\sigma}\right)^{p-1}\right)\left((u_\sigma)^p-(v_\sigma)^p\right)\, \di x\ge\\
\ge\int_\Omega \left(\left(\frac 1{v}\right)^{p-1}-\left(\frac 1{u}\right)^{p-1}\right)\left(u^p-v^p\right)\, \di x>0\,.
\end{eqnarray*}
This inequality contradicts \eqref{abs}, and the claim is proved.
\vskip.2cm
\noindent \underline{Item (\ref{(c)})}
\vskip.3cm
The claim follows immediately from the definition of $Q_p(\alpha, \Omega)$.

\vskip.2cm
\noindent \underline{Item (\ref{(d)})}
\vskip.3cm
The claim follows immediately from P\'olya-Szeg\H o principle given by Theorem \ref{polya_szego} and property \eqref{inv}. 
}
\end{proof}
Now we prove some properties of $Q_p(\alpha, \Omega)$ when $\Omega$ is fixed and the parameter $\alpha$ varies.
\begin{proposition}\label{propa}
Let $\Omega\subset \R^N$, with finite measure. Then:
\begin{enumerate}[\rm(a)]
\item\label{(aa)} if $v_{(\alpha)}$ denotes the solution to problem \eqref{veq.0} for a given value of the parameter $\alpha\in (-\infty,\lambda_1(\Omega))$, we have:
\begin{equation*}
v_{(\alpha)}(x)\le v_{(\beta)}(x), \qquad x\in \Omega,\>-\infty<\alpha<\beta<\lambda_1(\Omega);
\end{equation*}
\item\label{(ab)} $\quad\displaystyle Q_p(\alpha, \Omega)$ is increasing with respect to $\alpha$ and, if $v$ solves \eqref{veq.0}, it holds
\[\dfrac\di{\di\alpha}Q_p(\alpha,\Omega)=\int_\Omega |v(x)|^p\, \di x;
\]
\item\label{(ac)} $\quad\displaystyle\lim_{\alpha\rightarrow-\infty} Q_p(\alpha,\Omega)=0;$
\item\label{(ad)} $\quad\displaystyle\lim_{\alpha\rightarrow \lambda_1(\Omega)^-} Q_p(\alpha,\Omega)=+\infty.$
\end{enumerate}
\end{proposition}
\begin{proof} We prove in sequence the various items.
\vskip.2cm
\noindent \underline{Item (\ref{(aa)})}
\vskip.3cm
The assertion is natural, since increasing the parameter $\alpha$ increases
  the source term $\alpha v^{p-1}+1$ whenever $v\ge0$. We make this comparison
  rigorous by proving that the set where the solution corresponding to the
  smaller parameter exceeds the one corresponding to the larger parameter has
  zero measure.

The claim can be proved proceeding as in the proof of item (\ref{(b)}) of Proposition \ref{prop}.
Let us fix $-\infty<\alpha<\beta<\lambda_1(\Omega)$ and let us put
\begin{equation*}
u=v_{(\alpha)},\qquad w=v_{(\beta)}.
\end{equation*}
For $0<\sigma<k<+\infty$ we consider the test functions
\begin{equation*}
\varphi_1=\frac{\left(\bigl( u+\sigma\bigr)^p-{\bigl( w+\sigma) \bigr)^p}\right)_+}{(u+\sigma)^{p-1}},\qquad
\varphi_2=\frac{\left(\bigl( u+\sigma\bigr)^p-{\bigl( w+\sigma\bigr)^p}\right)_+}{(w+\sigma)^{p-1}},
\end{equation*}
where, for $s\in\R$, we use the notation   $s_+=\max\{s,0\}$. It is immediate to observe that $\varphi_1, \varphi_2\in
W_{0}^{1,p} ( \Omega)$, so we can use $\varphi_1$ in the equation satisfied by $u$ and $\varphi_2$ in the equation satisfied by $w$. Using the notation $u_\sigma=u+\sigma$, $w_\sigma=w+\sigma$ and $E=\{x:  u_\sigma(x))>  w_\sigma(x)\}$, we obtain
\begin{eqnarray*}
\int_{E}|Du|^p\, \di x-p\int_{E}|Du|^{p-2}DuD\bigl(w_\sigma\bigr)\frac{\bigl(w_\sigma\bigr)^{p-1}}{(u_\sigma)^{p-1}}\, \di x+
\\
+(p-1)\int_{E}|Du|^{p}\frac{\bigl(w_\sigma\bigr)^{p}}{(u_\sigma)^{p}}\, \di x=\\
=\alpha\int_{E} \left(\frac u{u_\sigma}\right)^{p-1}\left(\bigl(u_\sigma\bigr)^p-\bigl(w_\sigma\bigr)^p\right)\, \di x+\int_{E} \left(\frac{(u_\sigma)^p-\bigl(w_\sigma\bigr)^p}{(u_\sigma)^{p-1}}\right)\, \di x,
\end{eqnarray*}
\begin{eqnarray*}
\int_{E}|Dw|^p \di x-p\int_{E}|Dw|^{p-2}DwD\bigl(u_\sigma\bigr)\frac{\bigl(u_\sigma\bigr)^{p-1}}{(w_\sigma)^{p-1}}\, \di x+\\
+(p-1)\int_{E}|Dw|^{p}\frac{\bigl(u_\sigma\bigr)^{p}}{(w_\sigma)^{p}}\, \di x=\\
=\beta\int_{E} \left(\frac w{w_\sigma}\right)^{p-1}\left(\bigl(w_\sigma\bigr)^p-\bigl(u_\sigma\bigr)^p\right)\, \di x+\int_{E} \left(\frac{(w_\sigma)^p-\bigl(u_\sigma\bigr)^p}{(w_\sigma)^{p-1}}\right)\, \di x.
\end{eqnarray*}
Summing the above equalities and using Young inequalities
\begin{equation*}
\left|p|Du|^{p-2}DuD\bigl(w_\sigma\bigr)\frac{\bigl(w_\sigma\bigr)^{p-1}}{(u_\sigma)^{p-1}}\right|
\le |D\bigl(w_\sigma\bigr)|^p+(p-1)|Du|^{p}\frac{\bigl(w_\sigma\bigr)^{p}}{(u_\sigma)^{p}}
\end{equation*}
\begin{equation*}
\left|p|Dw|^{p-2}DwD\bigl(u_\sigma\bigr)\frac{\bigl(u_\sigma\bigr)^{p-1}}{(w_\sigma)^{p-1}}\right|
\le |D\bigl(u_\sigma\bigr)|^p+(p-1)|Dw|^{p}\frac{\bigl(u_\sigma\bigr)^{p}}{(w_\sigma)^{p}}
\end{equation*}
we obtain
\begin{align}\label{Tk}
\alpha\int_{E} \left(\left(\frac u{u_\sigma}\right)^{p-1}-\left(\frac w{w_\sigma}\right)^{p-1}\right)\left(\bigl(u_\sigma\bigr)^p-\bigl(w_\sigma\bigr)^p\right)\, \di x&\\
+(\beta-\alpha)\int_{E} \left(\frac w{w_\sigma}\right)^{p-1}\left(\bigl(w_\sigma\bigr)^p-\bigl(u_\sigma\bigr)^p\right)\, \di x&\notag\\
+\int_{E} \left(\frac{(u_\sigma)^p-\bigl(w_\sigma\bigr)^p}{(u_\sigma)^{p-1}}\right)\, \di x
+\int_{E} \left(\frac{(w_\sigma)^p-\bigl(u_\sigma\bigr)^p}{(w_\sigma)^{p-1}}\right)\, \di x&\ge0\notag
\end{align}
Then we get:
\begin{eqnarray*}
\alpha\int_{u>w} \left(\left(\frac u{u_\sigma}\right)^{p-1}-\left(\frac w{w_\sigma}\right)^{p-1}\right)\left((u_\sigma)^p-(w_\sigma)^p\right)\, \di x+\\
+(\beta-\alpha)\int_{u>w} \left(\frac w{w_\sigma}\right)^{p-1}\left((w_\sigma)^p-(u_\sigma)^p\right)\, \di x\ge\\
\ge\int_{u>w} \left(\left(\frac 1{w_\sigma}\right)^{p-1}-\left(\frac 1{u_\sigma}\right)^{p-1}\right)\left((u_\sigma)^p-(w_\sigma)^p\right)\, \di x.
\end{eqnarray*}
It is clear that 
\begin{equation*}
\lim_{\sigma\rightarrow0^+}\int_{u>w} \left(\left(\frac u{u_\sigma}\right)^{p-1}-\left(\frac w{w_\sigma}\right)^{p-1}\right)\left((u_\sigma)^p-(w_\sigma)^p\right)\, \di x=0,
\end{equation*}
and
\begin{equation*}
\lim_{\sigma\rightarrow0^+}\int_{u>w} \left(\frac w{w_\sigma}\right)^{p-1}\left((w_\sigma)^p-(u_\sigma)^p\right)\, \di x=\int_{u>w} \left(w^p-u^p\right)\, \di x\le0,
\end{equation*}
then
\begin{equation}\label{absa}
\liminf_{\sigma\rightarrow0^+}\int_{u>w} \left(\left(\frac 1{w_\sigma}\right)^{p-1}-\left(\frac 1{u_\sigma}\right)^{p-1}\right)\left((u_\sigma)^p-(w_\sigma)^p\right)\, \di x
\le0,
\end{equation}
where the integrand is nonnegative, that is,
\begin{eqnarray*}
\left(\left(\frac 1{w_\sigma}\right)^{p-1}-\left(\frac 1{u_\sigma}\right)^{p-1}\right)\left((u_\sigma)^p-(w_\sigma)^p\right)
\ge0.
\end{eqnarray*}
If $E=\{x:  u(x))>  w(x)\}$ $u> w$ is a set of positive measure, by Fatou lemma we have
\begin{eqnarray*}
\liminf_{\sigma\rightarrow0^+}\int_{u>w} \left(\left(\frac 1{w_\sigma}\right)^{p-1}-\left(\frac 1{u_\sigma}\right)^{p-1}\right)\left((u_\sigma)^p-(w_\sigma)^p\right)\, \di x\ge\\
\ge\int_{u>w} \left(\left(\frac 1{w}\right)^{p-1}-\left(\frac 1{u}\right)^{p-1}\right)\left(u^p-w^p\right)\, \di x>0,
\end{eqnarray*}
which contradicts \eqref{absa}, and the claim is proved.
\vskip.2cm
\noindent \underline{Item (\ref{(ab)})}
\vskip.3cm

The monotonicity of $Q_p(\alpha, \Omega)$ with respect to $\alpha$ can be proven from the definition.
In order to prove the differentiation formula we firstly show that, for every $\alpha\in(-\infty, \lambda_1(\Omega))$, using the notation of item (\ref{(aa)}), we have
\begin{equation}\label{conve}
v_{(\alpha+\varepsilon)}\rightarrow v_{(\alpha)}, \quad \text{strongly in }L^p(\Omega ),\text{  as }\varepsilon\rightarrow0.
\end{equation}
Let us put
$$
u_\varepsilon =v_{(\alpha+\varepsilon)}, \qquad u = v_{(\alpha)}.$$
Using equation \eqref{veq.0} satisfied by $u_\varepsilon$, we have, for $\delta>0$ suitably small and  for a suitable constant $C(\delta)>0$,
\begin{eqnarray*}
\int_\Omega|Du_\varepsilon|^p\, \di x=(\alpha+\varepsilon)\int_\Omega |u_\varepsilon|^p\, \di x+p\int_\Omega u_\varepsilon\, \di x\\
\le\left(\frac{\alpha+\varepsilon}{\lambda_1(\Omega)}+\delta\right)\int_\Omega|Du_\varepsilon|^p\, \di x+C(\delta)|\Omega|.
\end{eqnarray*}
This means that $u_\varepsilon$ is bounded in $W^{1,p}(\Omega)$. Then there exists a subsequence $u_{\varepsilon_h}$ which strongly converges in $L^p(\Omega)$ for $\varepsilon$ which goes to 0. Actually, in view of the monotonicity proven in item
(\ref{(aa)}), we can say that the whole sequence $u_\varepsilon$ is such that
\begin{equation}\label{convu}
u_\varepsilon\rightarrow \bar u, \quad \text{strongly in }L^p(\Omega),\text{  as }\varepsilon\rightarrow0,
\end{equation}
for some $\bar u\in L^p(\Omega)$. A result contained, for example, in  \cite{BoccardoMurat} allows us to get the almost everywhere convergence of $Du_\varepsilon$, then we can pass to the limit as $\varepsilon \rightarrow0$ in the equation satisfied by $u_\varepsilon$. In view of the uniqueness stated in Proposition \ref{prop}, item (\ref{(b)}), we have that $\bar u=u$, then \eqref{conve} is proved.

In order to prove the differentiation formula, we observe that, using the definition of $Q_p(\alpha,\Omega)$ and of $F_\alpha(w)$ given in \eqref{QQ} and \eqref{F}, respectively, it holds:
\begin{equation*}
F_{\alpha+\varepsilon}(u_\varepsilon)-F_{\alpha}(u_\varepsilon)\ge Q_p(\alpha+\varepsilon,\Omega)-Q_p(\alpha,\Omega)
\ge F_{\alpha+\varepsilon}(u)-F_{\alpha}(u)
\end{equation*}
that is,
\begin{equation}\label{FQa}
\varepsilon\int_\Omega |u_\varepsilon|^p\, \di x\ge Q_p(\alpha+\varepsilon,\Omega)-Q_p(\alpha,\Omega)
\ge \varepsilon\int_\Omega |u|^p\, \di x.
\end{equation}
Taking into account \eqref{conve}, inequalities \eqref{FQa} imply
$$\lim_{\varepsilon\rightarrow0} \frac{Q_p(\alpha+\varepsilon,\Omega)-Q_p(\alpha,\Omega)}\varepsilon=
\int_\Omega |u|^p\, \di x
$$
and the claim is proved.
\vskip.2cm
\noindent \underline{Item (\ref{(ac)})}
\vskip.3cm
Since \eqref{Q} holds true, we prove that 
\begin{equation}\label{limQ}
\displaystyle\lim_{\alpha\rightarrow-\infty} \int_\Omega v_{(\alpha)}(x)\, \di x=0  \,.
\end{equation}
To this aim we choose $v_{(\alpha)}$ as test function in \eqref{veq.0} and we get
\begin{equation*}
\frac 1\alpha \int_\Omega |D v_{(\alpha)}|^p\, \di x\,=
\int_\Omega |v_{(\alpha)}(x)|^p\, \di x
+\frac 1\alpha \int_\Omega v_{(\alpha)}(x)\, \di x\,.
\end{equation*}
Since $\alpha<0$, the right hand side of the above equality is negative. Applying H\"older inequality we then get
\begin{align*}
\int_\Omega | v_{(\alpha)}|^p\, \di x &\le-\frac 1\alpha\int_\Omega v_{(\alpha)}(x)\, \di x\\
\\
\displaystyle\quad&\le -\frac 1\alpha |\Omega|^{1-\frac 1p}
\left ( \int_\Omega |v_{(\alpha)}(x)|^p\, \di x \right)^\frac{1}{p}.
\end{align*}
We deduce
\begin{equation*}
\int_\Omega | v_{(\alpha)}|^p\, \di x \le \frac 1{(-\alpha)^{\frac p{p-1}}} |\Omega|
\end{equation*}
and therefore
\begin{equation}\label{limQ2}
\lim_{\alpha \to -\infty}\int_\Omega | v_{(\alpha)}|^p\, \di x =0\,.
\end{equation}
On the other hand, by using H\"older inequality, we get
\begin{equation}\label{limQ3}
\left |\int_\Omega  v_{(\alpha)}\, \di x \right |
\le 
 |\Omega|^{1-\frac 1p}
\left (\int_\Omega | v_{(\alpha)}|^p\, \di x\right)^\frac 1p.
\end{equation}
Combining \eqref{limQ2} and \eqref{limQ3}, the assertion follows. 
\vskip.2cm
\noindent \underline{Item (\ref{(ad)})}
\vskip.3cm

We note that, by the monotonicity of $Q_p(\alpha, \Omega)$ stated in item (\ref{(ab)}), the limit
\[\lim_{\alpha\rightarrow\lambda_1(\Omega)^-}Q_p(\alpha, \Omega)
\]
exists, possibly infinite. 
We can use as test function $w=ku$, where $k$ is an arbitrary positive constant and $u$ is a positive eigenfunction of problem \eqref{eq.0}, obtaining
\begin{align*}
\displaystyle Q_p(\alpha, \Omega)&\ge-\int_\Omega|D(ku)|^p\, \di x+\alpha\int_\Omega |ku|^p\, \di x+p\int_\Omega ku\, \di x=\\
\\
\displaystyle\quad&= \bigl(\alpha-\lambda_1(\Omega)\bigr)k^p\int_\Omega |u|^p\, \di x+pk\int_\Omega u\, \di x.
\end{align*}
Letting $\alpha\rightarrow\lambda_1(\Omega)^-$, we have
\[\lim_{\alpha\rightarrow\lambda_1(\Omega)^-}Q_p(\alpha, \Omega)\ge pk\int_\Omega u\, \di x
\]
and from the arbitrariness of $k$ the claim follows.
\end{proof}

When $\Omega$ is a ball, all the results stated in Proposition \ref{prop} hold true, but some further properties about the behavior of $Q_p(\alpha, \Omega)$ with respect to the choice of the ball can be added. So, we put, for $r>0$,
\begin{equation*}
B_r= \{x  \in\RR : |x| < r\}
\end{equation*}
and we introduce the function of two variables
\begin{equation}\label{Qrad}
Q_p^\sharp(\alpha,r)=Q_p(\alpha,B_{r}), 
\end{equation}
defined on the following set
$$D = \{(\alpha, r) : \alpha \le 0, r > 0\} \cup \{(\alpha, r) : \alpha > 0, 0 < r < g(\alpha)\}
$$
being
$$g(\alpha)=\left(\frac{\lambda_1(B_1)}\alpha\right)^{\frac1{p}}.
$$
Indeed, if $\alpha>0$ and $0<r<g(\alpha)$, we have 
\[
\alpha< r^{-p}\lambda_{1}(B_{1})=\lambda_{1}(B_{r})
\]
and the value $Q_p^\sharp(\alpha,r)$ is finite.

Let us also observe that a simple scaling argument shows that
\begin{equation}\label{Qris}
Q_p^\sharp(\alpha,r)=r^{\frac{Np-N+p}{p-1}}Q_p^\sharp(\alpha r^{p},1), \qquad (\alpha,r) \in D.
\end{equation}
We now prove some results about the behavior of $Q_p^\sharp(\alpha,r)$ with respect to $r$.
\begin{proposition}
\label{finite}
Let $Q_p^\sharp(\alpha,r)$ be the function defined in \eqref{Qrad}, we have:
\vskip.2cm
\noindent- for any fixed $\alpha\le0$, $Q_p^\sharp(\alpha,r)$ is finite for every $r>0$ and
\begin{equation}\label{itlim}
\lim_{r\rightarrow+\infty}Q_p^\sharp(\alpha,r)=+\infty;
\end{equation}
\noindent- for any fixed $\alpha>0$, $Q_p^\sharp(\alpha,r)$  is finite if and only if
\begin{equation}\label{t}
0<r<{\bar r}\equiv\left(\frac{\lambda_1(B_1)}\alpha\right)^{\frac1{p}};
\end{equation}
furthermore
\begin{equation}\label{tlim}
\lim_{r\rightarrow\,{\bar r^-}}Q_p^\sharp(\alpha,r)=+\infty.
\end{equation}
\end{proposition}

\begin{proof}
By Proposition \ref{prop}, item \ref{(a)}, and by the scaling law
\[
\lambda_1(B_r)=r^{-p}\lambda_1(B_1),
\]
the quantity $Q_p^\sharp(\alpha,r)$ is finite if and only if
\[
\alpha<\lambda_1(B_r)=r^{-p}\lambda_1(B_1).
\]
Hence, if $\alpha \le 0$, 
$Q_p^\sharp(\alpha,r)$ is finite for every
$r>0$. If $\alpha>0$, it is finite if and only if
\[
0<r<\bar r:=\left(\frac{\lambda_1(B_1)}{\alpha}\right)^{1/p}.
\]
Let us prove \eqref{itlim}. We first assume that 
$\alpha<0$. For
$r>1$, set
\[
c_\alpha=\left(-\frac1\alpha\right)^{\frac1{p-1}}
\]
and define
\[
w_r(x)=
\begin{cases}
c_\alpha, & |x|<r-1,\\[2mm]
c_\alpha(r-|x|), & r-1\le |x|<r,\\[2mm]
0, & |x|\ge r.
\end{cases}
\]
Then 
$ w_r\in W_0^{1,p}(B_r)$, and therefore
\[
F_\alpha(w_r)\le Q_p^\sharp(\alpha,r).
\]
Since
\[
\alpha c_\alpha^p+p c_\alpha=(p-1)c_\alpha,
\]
we have
\[
F_\alpha(w_r)
\ge
(p-1)c_\alpha |B_{r-1}|
-
c_\alpha^p |B_r\setminus B_{r-1}|
+
\int_{B_r\setminus B_{r-1}}
\bigl(\alpha w_r^p+p w_r\bigr)\,\di x .
\]
The last two terms are bounded from below by a quantity of order
$r^{N-1}$, while the first term is of order $r^N$.
Hence
\[
\lim_{r\to+\infty}F_\alpha(w_r)=+\infty.
\]
Consequently we have
\[
\lim_{r\to+\infty}Q_p^\sharp(\alpha,r)=+\infty .
\]
If $\alpha=0$, for $r>1$ and 
$k>0$, define
\[
w_{r,k}(x)=
\begin{cases}
k, & |x|<r-1,\\[2mm]
k(r-|x|), & r-1\le |x|<r,\\[2mm]
0, & |x|\ge r.
\end{cases}
\]
Then
\[
F_0(w_{r,k})
=
-\int_{B_r}|\nabla w_{r,k}|^p\,\di x
+
p\int_{B_r}w_{r,k}\,\di x .
\]
In particular,
\[
F_0(w_{r,k})
\ge
-k^p |B_r\setminus B_{r-1}|
+
pk |B_{r-1}|.
\]
Since the first term is of order $r^{N-1}$, whereas the second one is
of order $r^N$, we get
\[
\lim_{r\to+\infty}F_0(w_{r,k})=+\infty .
\]
Thus also
\[
\lim_{r\to+\infty}Q_p^\sharp(0,r)=+\infty .
\]
It remains to prove \eqref{tlim}.
Let $\alpha>0$ and let
\[
\bar r=\left(\frac{\lambda_1(B_1)}{\alpha}\right)^{1/p}.
\]
Using the scaling formula \eqref{Qris}, we have
\[
Q_p^\sharp(\alpha,r)
=
r^{\frac{Np-N+p}{p-1}}
Q_p^\sharp(\alpha r^p,1).
\]
As $r\to\bar r^-$, one has
\[
\alpha r^p\to \alpha \bar r^p=\lambda_1(B_1).
\]
Therefore, by Proposition \ref{propa}, item \ref{(ad)}, applied to
the unit ball $B_1$,
\[
Q_p^\sharp(\alpha r^p,1)\to+\infty.
\]
Since
\[
r^{\frac{Np-N+p}{p-1}}\to
\bar r^{\frac{Np-N+p}{p-1}}>0,
\]
we obtain
\[
\lim_{r\to\bar r^-}Q_p^\sharp(\alpha,r)=+\infty.
\]
The proof is complete.
\end{proof}

\begin{proposition}\label{derQ}
For any fixed $\alpha$, the function $Q_p^\sharp(\alpha,r)$ 
defined in \eqref{Qrad}
is strictly increasing with respect to $r$
in its interval of definition. Moreover,
if $v$ denotes the generalized torsion function in $B_r$, then
\begin{equation}
\label{shapeQ}
\frac{\partial}{\partial r}Q_p^\sharp(\alpha,r)
=
(p-1)\int_{\partial B_r}|\nabla v|^p\,\di\mathcal H^{N-1}.
\end{equation}
Equivalently, since $v$ is radial,
\begin{equation}\label{shapeQrad}
\frac{\partial}{\partial r}Q_p^\sharp(\alpha,r)
=
(p-1)N\omega_N r^{N-1}|v'(r)|^p .
\end{equation}
\end{proposition}

\begin{proof}
Let us fix $(\alpha,r)$ in the domain of definition of $Q_p^\sharp$. We denote by
$v=v_{\alpha,r}$ the solution of
\[
\left\{
\begin{array}{rcll}
-\Delta_p v & = & \alpha v^{p-1}+1 & \text{in } B_r,\\[3mm]
v & = & 0 & \text{on } \partial B_r.
\end{array}
\right.
\]
Since the problem is invariant under rotations and the solution is unique,
the generalized torsion function $v$
is a radial function.

We first recall the scaling of $Q_p^\sharp$. Let
\[
\beta=\alpha r^p
\]
and let $u=u_\beta$ be the solution of
\[
\left\{
\begin{array}{rcll}
-\Delta_p u & = & \beta u^{p-1}+1 & \text{in } B_1,\\[3mm]
u & = & 0 & \text{on } \partial B_1.
\end{array}
\right.
\]
Then
\[
v(x)=r^{\frac p{p-1}}u\left(\frac{x}{r}\right),
\]
and hence
\begin{equation}\label{scalingQ}
Q_p^\sharp(\alpha,r)
=
r^{\frac{Np-N+p}{p-1}}Q_p^\sharp(\alpha r^p,1).
\end{equation}
Set
\[
a=\frac{Np-N+p}{p-1}.
\]
By Proposition \ref{propa}, applied to the ball $B_1$, the map
$\beta\mapsto Q_p^\sharp(\beta,1)$ is differentiable and
\[
\frac{\partial}{\partial\beta}Q_p^\sharp(\beta,1)
=
\int_{B_1}u_\beta^p\,\di x.
\]
Differentiating \eqref{scalingQ} with respect to $r$ gives
\begin{align}
\frac{\partial}{\partial r}Q_p^\sharp(\alpha,r)
&=
a r^{a-1}Q_p^\sharp(\beta,1)
+
p\alpha r^{a+p-1}
\frac{\partial}{\partial\beta}Q_p^\sharp(\beta,1) \notag\\
&=
\frac{a}{r}Q_p^\sharp(\alpha,r)
+
\frac{p\alpha}{r}\int_{B_r}v^p\,\di x. \label{derivative_scaling}
\end{align}
Since
\[
Q_p^\sharp(\alpha,r)=(p-1)\int_{B_r}v\,\di x,
\]
we obtain
\begin{equation}\label{derivative_intermediate}
\frac{\partial}{\partial r}Q_p^\sharp(\alpha,r)
=
\frac1r
\left[
(Np-N+p)\int_{B_r}v\,\di x
+
p\alpha\int_{B_r}v^p\,\di x
\right].
\end{equation}

It remains to identify the right-hand side with the boundary term. To this aim, we use the Pohozaev identity for
\[
-\Delta_p v=\alpha v^{p-1}+1.
\]
Since $v=0$ on $\partial B_r$, we have
\begin{equation}\label{pohozaev}
\frac{N-p}{p}\int_{B_r}|\nabla v|^p\,\di x
-
N\int_{B_r}\left(\frac{\alpha}{p}v^p+v\right)\,\di x
+
\frac{p-1}{p}r\int_{\partial B_r}|\nabla v|^p\,\di\mathcal H^{N-1}
=0.
\end{equation}
On the other hand, testing the equation with $v$ gives
\begin{equation}\label{testv}
\int_{B_r}|\nabla v|^p\,\di x
=
\alpha\int_{B_r}v^p\,\di x
+
\int_{B_r}v\,\di x.
\end{equation}
Combining \eqref{pohozaev} and \eqref{testv}, we get
\[
(p-1)r\int_{\partial B_r}|\nabla v|^p\,\di\mathcal H^{N-1}
=
p\alpha\int_{B_r}v^p\,\di x
+
(Np-N+p)\int_{B_r}v\,\di x.
\]
Substituting this identity into \eqref{derivative_intermediate} yields
\[
\frac{\partial}{\partial r}Q_p^\sharp(\alpha,r)
=
(p-1)\int_{\partial B_r}|\nabla v|^p\,\di\mathcal H^{N-1},
\]
which is \eqref{shapeQ}.

Finally, since $v$ is positive in $B_r$ and vanishes on $\partial B_r$, the Hopf boundary lemma gives
\[
|\nabla v|>0 \quad \text{on } \partial B_r.
\]
Therefore
\[
\frac{\partial}{\partial r}Q_p^\sharp(\alpha,r)>0,
\]
and $Q_p^\sharp(\alpha,r)$ is strictly increasing with respect to $r$.
\end{proof}

\begin{remark}
Formula \eqref{shapeQ} is the 
shape derivative of
$Q_p(\alpha,\Omega)$ along the family of balls. In particular, it implies
\[
\frac{\partial}{\partial r}Q_p^\sharp(\alpha,r)>0
\]
whenever $Q_p^\sharp(\alpha,r)$ is finite.
\end{remark}

An immediate consequence of the above results is the following proposition.

\begin{proposition}\label{prad}
For any open set $\Omega\subset \R^N$ 
with 
finite measure, and for every
$-\infty<\alpha<\lambda_1(\Omega)$, 
there exists a unique
$r(\alpha)>0$ such that
\begin{equation}
\label{cnd}
Q_p(\alpha,\Omega)=Q_p(\alpha,B_{r(\alpha)}).
\end{equation}
Moreover,
\[
B_{r(\alpha)}\subset\Omega^\sharp .
\]
Finally, if $\alpha>0$, then
\[
r(\alpha)<\bar r,
\]
where $\bar r$ is defined by 
$\lambda_1(B_{\bar r})=\alpha$.
\end{proposition}

\begin{proof}
Let us first prove existence and uniqueness. 
By Proposition \ref{derQ},
the map
\[
r\mapsto Q_p^\sharp(\alpha,r)
\]
is strictly increasing on its interval of definition.

Moreover, by the scaling formula \eqref{Qris},
\[
Q_p^\sharp(\alpha,r)
=
r^{\frac{Np-N+p}{p-1}}Q_p^\sharp(\alpha r^p,1).
\]
It follows that
\[
\lim_{r\to0^+}Q_p^\sharp(\alpha,r)=0.
\]
Indeed, $\alpha r^p \to 0$ 
and 
$Q_p^\sharp(\beta,1)$ remains finite
for $\beta$ near $0$.

If $\alpha\le 0$, Proposition \ref{finite}
gives
\[
\lim_{r\to+\infty}Q_p^\sharp(\alpha,r)=+\infty.
\]
Hence $Q_p^\sharp(\alpha,\cdot)$
maps 
$(0,+\infty)$ 
onto
$(0,+\infty)$.

If $\alpha>0$, then $Q_p^\sharp(\alpha,r)$ is finite  for
\[
0<r<\bar r,
\qquad
\lambda_1(B_{\bar r})=\alpha,
\]
and Proposition \ref{finite} gives
\[
\lim_{r\to\bar r^-}Q_p^\sharp(\alpha,r)=+\infty.
\]
Hence 
$Q_p^\sharp(\alpha,\cdot)$ 
maps 
$(0,\bar r)$ 
onto
$(0,+\infty)$.
Since 
$Q_p(\alpha,\Omega)\in(0,+\infty)$, 
there exists a unique
$r(\alpha)$ such that \eqref{cnd} holds.

It remains to prove that
$B_{r(\alpha)}\subset\Omega^\sharp$. 
Let
$\Omega^\sharp=B_R$.
If
\[
\alpha<\lambda_1(B_R)=\lambda_1(\Omega^\sharp),
\]
then Proposition \ref{prop}, item \ref{(d)}, gives
\[
Q_p(\alpha,\Omega)\le Q_p(\alpha,\Omega^\sharp)
=
Q_p^\sharp(\alpha,R).
\]
Since $Q_p^\sharp(\alpha,\cdot)$ is strictly increasing, from
\[
Q_p^\sharp(\alpha,r(\alpha))=Q_p(\alpha,\Omega)
\]
we deduce
\[
r(\alpha)\le R.
\]
On the other hand, if
\[
\alpha\ge \lambda_1(B_R),
\]
then necessarily $\alpha>0$. Let $\bar r$ 
be such that
\[
\lambda_1(B_{\bar r})=\alpha.
\]
Since $r\mapsto\lambda_1(B_r)$ 
is strictly decreasing, the inequality
\[
\alpha\ge \lambda_1(B_R)
\]
implies
\[
\bar r\le R.
\]
Moreover, by the finiteness of 
$Q_p^\sharp(\alpha,r(\alpha))$, we have
\[
r(\alpha)<\bar r.
\]
Therefore
\[
r(\alpha)<\bar r \le R.
\]
In both cases, $r(\alpha)\le R$, namely
\[
B_{r(\alpha)}\subset\Omega^\sharp .
\]
\end{proof}

\section{Comparison results}\label{s3}

In this section we show that one can use standard symmetrization arguments which go back to Talenti results \cite{T}, \cite{T2}, in order to prove a comparison result which allows us to estimate the generalized torsion function introduced in Section \ref{s2}. Let us observe that similar results have been obtained, for example, in \cite{C1}, \cite{C2} when $p=2$, or in \cite{AFT} when $p>1$.
\begin{theorem}\label{comp}
For a fixed $\alpha\in(-\infty,\lambda_1(\Omega))$, let $v$ be the solution to problem \eqref{veq.0} and let $r>0$ be such that
$Q_p(\alpha,\Omega)= Q_p^\sharp(\alpha,r)$. If $\bar v$ is the solution to problem
\begin{equation} \label{veq}
\left\{
\begin{array}
[c]{lll}%
-\dive( |D\bar v|^{p-2}D\bar v)=\alpha  |\bar v|^{p-2}\bar v+1 & & \text{in }
B_r,\\
\\
\bar v=0 & & \text{on }\partial B_r,
\end{array}
\right. %
\end{equation}
then, if $\max\{1,p-1\}\le m<\infty$, for every $s\le |\Omega|$ it holds
\begin{equation}\label{qq}
\int_0^{s}\bigl(v^*(\sigma)\bigr)^m\, \di \sigma \le \int_0^{s}\bigl(\bar v^*(\sigma)\bigr)^m\, \di \sigma,
\end{equation}
$\bar v^*(s)$ is extended to 0 for
$s>|B_r|$.
\end{theorem}
\begin{remark}\label{REM}
As it will be clear in the proof, when $\alpha\ge0$, inequality \eqref{qq} holds true for every $1\le m<\infty$.
\end{remark}
\begin{proof}[Proof of Theorem \ref{comp}]
By standard arguments (see, for example, \cite{T2}) we can verify that for $v^*$ the following inequality holds true
\begin{equation}\label{uprob}
\bigl((-v^*)'(s)\bigr)^{p-1}
\le \frac\alpha{N^p\omega_N^{\frac pN}s^{p-\frac pN}} \int_0^s\bigl(v^*(\sigma)\bigr)^{p-1} \,\di\sigma +\frac1{N^p\omega_N^{\frac pN}s^{p-1-\frac pN}}\,, \quad\hbox{ a.e. in }
(0,|\Omega|),
\end{equation}
while $\bar v^*$ satisfies
\begin{equation}\label{vprob}
\bigl((-\bar v^*)'(s)\bigr)^{p-1}
= \frac\alpha{N^p\omega_N^{\frac pN}s^{p-\frac pN}} \int_0^s\bigl(\bar v^*(\sigma)\bigr)^{p-1} \,\di\sigma +\frac1{N^p\omega_N^{\frac pN}s^{p-1-\frac pN}}\,, \quad\hbox{ a.e. in }
(0,|B_r|).
\end{equation}

Let us observe that, in view of Proposition \ref{prop}, items (\ref{(c)}), (\ref{(d)}), it holds
\begin{equation}\label{compH}
|B_r| \le 
|\Omega|.
\end{equation}
Furthermore,
the equality $Q_p(\alpha,\Omega)= Q_p^\sharp(\alpha,r)$ implies
\begin{equation}\label{obs}
\int_0^{|\Omega|}v^*(\sigma)\,\di\sigma=\int_0^{|B_r|}\bar v^*(\sigma) \,\di\sigma,
\end{equation}
then it is well defined
$$\bar s=\sup\{s>0:v^*(s)=\bar v^*(s)\}\le|B_r|.
$$
When $\alpha\ge0$, we put
\begin{equation}\label{v}
w(s)=\left\{
   \begin{array}{ll}
\max\{v^*(s),\bar v^*(s)\},&0< s<\bar s\\
&\\
\bar v^*(s),&\text{otherwise}
\end{array},
 \right.
\end{equation}
then, from \eqref{uprob} and \eqref{vprob} 
we know that for a.e. $s\in (0,|B_r|)$,
\begin{equation}\label{wprob}
 \left\{
   \begin{array}{lll}\displaystyle
N^p\omega_N^{\frac pN}s^{p-\frac pN}\bigl(-w'(s)\bigr)^{p-1}
\le \alpha \int_0^s\bigl(w(\sigma)\bigr)^{p-1} \,\di\sigma +s\,,\\
&&\\
w(s)\ge\bar v^*(s)\,.
\end{array}
 \right.
\end{equation}
Multiplying the first inequality by $-w'(s)$ and integrating, we have
\begin{equation}\label{intW}
\int_0^{|B_r|}
\Bigl(N\omega_N^{\frac 1N}s^{1-\frac 1N}|w'(s)|\Bigr)
^{p}\,\di s
\le
\alpha\int_0^{|B_r|} |w(s)|^p\,\di s +\int_0^{|B_r|}w(s)\,\di s\,.
\end{equation}
We can also use $\tilde w(x)=w(\omega_N|x|^N)$ as a test function in $Q^\sharp_{p}(\alpha,r)$, to get
\begin{align*}
\displaystyle
Q^\sharp_{p}(\alpha,r)&\ge-\int_{B_r}|\tilde w'(x)|^{p}
\,\di x+\alpha
\int_{B_r}|\tilde w(x)|^{p}
\,\di x +p\int_{B_r} \tilde w(x)
\,\di x
\\
\\&=-\int_0^{|B_r|}
\Bigl(N\omega_N^{\frac 1N}s^{1-\frac 1N}|w'(s)|\Bigr)
^{p}\,\di s
+
\alpha\int_0^{|B_r|} |w(s)|^p\,\di s +p\int_0^{|B_r|}w(s)\,\di s\\
\\
\displaystyle
&\ge (p-1)\int_0^{|B_r|}w(s)\,\di s\ge
(p-1)\int_0^{|B_r|}\bar v^*(s)\,\di s
=
Q^\sharp_{p}(\alpha,r).
\end{align*}
By the characterization of the maximum we have $w=\bar v^*$, that is
\begin{equation*}
\left\{
   \begin{array}{lll}\displaystyle
\bar v^*(s)\ge v^*(s) \quad\text{in }(0,\bar s),\\
&&\\
v^*(s)\ge \bar v^*(s) \quad\text{otherwise}\,.
\end{array}
 \right.
\end{equation*}
Using again \eqref{obs}, we have
\begin{equation*}
\int_0^{s} v^*(\sigma)\,\di \sigma
\le
\int_0^{s} 
\bar v^*(\sigma)\,\di \sigma,\qquad s>0,
\end{equation*}
and, finally,  \eqref{qq} follows
by Proposition \ref{relazione}.

When $\alpha<0$, we consider the function
\begin{equation}\label{psi}
\Psi(s)=\int_0^{s}\left(\bigl(v^*(\sigma)\bigr)^{p-1}-\bigl(\bar v^*(\sigma)\bigr)^{p-1}\right)\, \di \sigma,\qquad0\le s\le|\Omega|.
\end{equation}
We claim that $\Psi$ cannot achieve a positive local maximum at any point $0<s<|\Omega|$. Let us observe that we can immediately exclude the points $|B_r|\le s\le|\Omega|$ because, as already observed, in those points $\Psi(s)$ is strictly increasing.

Suppose, by contradiction, that there exists $s_0$, with $0< s_0<|B_r|$, such that, for some $\delta>0$, it holds
\begin{equation}\label{aneg}
\Psi(s_0)=\max_{|s-s_0|<\delta}\Psi(s)>0,
\end{equation}
with $v^*(s_0)-\bar v^*(s_0)=0$.

From \eqref{uprob} and \eqref{vprob}, it follows, for a.e. $|s-s_0|<\delta$,
\begin{equation*}
(v^*)'(s)> (\bar v^*)'(s).
\end{equation*}
This means that the function $v^*(s)- \bar v^*(s)$ is strictly increasing, with $v^*(s_0)-\bar v^*(s_0)=0$, then \eqref{aneg} cannot hold.

We finally observe that there are two possibilities:

\begin{enumerate}[(\emph{\roman*})]
\item $\Psi(s)\le0$, $\forall s\in (0,|\Omega|)$;
\item $\Psi(s)>0$, for some $s\in (0,|\Omega|)$.
\end{enumerate}
In case \emph{(i)} there is nothing to prove, \eqref{qq} follows
by Proposition \ref{relazione}. 

In case \emph{(ii)} we have 
\begin{equation}\label{contra}
\max_{0\le s\le|\Omega|}\Psi(s)=\Psi(|\Omega|)>0,
\end{equation}
and there exists $\tilde s\in(0,|\Omega|)$ such that
\begin{equation}\label{tre}
\left\{
   \begin{array}{ll}\displaystyle
\Psi(s)\le0 &\quad\text{for }0\le s\le\tilde s,\\
&\\
\Psi(s)>0 &\quad\text{for }\tilde s< s\le|\Omega|,\\
&\\
v^*(s)\ge \bar v^*(s) &\quad\text{for }\tilde s< s\le|\Omega|\,.
\end{array}
 \right.
\end{equation}

If $1<p\le2$, we have $p-1\le1$, then the first and the third inequalities in \eqref{tre}, together with \eqref{obs} and Proposition \ref{relazione} imply \eqref{qq}.

If $p>2$, for $s<|\Omega|$, using the inequality $a-b\ge a^{2-p}(a^{p-1}-b^{p-1})/(p-1)$, $a,b>0$, we have:
\begin{align*}
\displaystyle\int_0^{s}\displaystyle\left(v^*(\sigma)-\bar v^*(\sigma)\right)\, &\di \sigma
\ge\int_0^{s}\left(\bigl(v^*(\sigma)\bigr)^{p-1}-\bigl(\bar v^*(\sigma)\bigr)^{p-1}\right)\frac{\bigl(v^*(\sigma)\bigr)^{2-p}}{p-1}\, \di \sigma=\\
\\
=&\displaystyle\frac{\bigl(v^*(s)\bigr)^{2-p}}{p-1}\Psi(s)-\int_0^{s}
\left(\frac{\bigl(v^*(\sigma)\bigr)^{2-p}}{p-1}\right)'\Psi(\sigma)\, \di \sigma\ge\\
\\
\ge&\displaystyle\frac{\bigl(v^*(s)\bigr)^{2-p}}{p-1}\Psi(s)-\left(\frac{\bigl(v^*(s)\bigr)^{2-p}}{p-1}-\frac{\bigl(v^*(0)\bigr)^{2-p}}{p-1}\right)\Psi(|\Omega|)=\\
\\
\ge&\displaystyle\frac{\bigl(v^*(0)\bigr)^{2-p}}{p-1}\Psi(|\Omega|)-\frac{\bigl(v^*(s)\bigr)^{2-p}}{p-1}\int_s^{|\Omega|}\left(\bigl(v^*(\sigma)\bigr)^{p-1}-\bigl(\bar v^*(\sigma)\bigr)^{p-1}\right)\, \di \sigma.
\end{align*}
The last term above vanishes as $s\rightarrow|\Omega|$ because
\begin{equation*}
\frac1{\bigl(v^*(s)\bigr)^{p-2}}\int_s^{|\Omega|}\left(\bigl(v^*(\sigma)\bigr)^{p-1}-\bigl(\bar v^*(\sigma)\bigr)^{p-1}\right)\, \di \sigma\le v^*(s)(|\Omega|-s).
\end{equation*}
Thus, as a consequence of \eqref{contra}, we have
\begin{equation*}
\int_0^{|\Omega|}\left(v^*(\sigma)-\bar v^*(\sigma)\right)\, \di \sigma
\ge\displaystyle\frac{\bigl(v^*(0)\bigr)^{2-p}}{p-1}\Psi(|\Omega|),
\end{equation*}
which contradicts \eqref{obs}, so \eqref{contra} cannot be verified and \eqref{qq} holds true for $m=p-1$. Proposition \ref{relazione} then applies.
\end{proof}

\section{Main result}\label{s4}

In this section we prove the main result of the paper which can be stated as follows.

\begin{theorem}\label{final}
For any $\alpha\in(-\infty,\lambda_1(\Omega))$ and for any open set $\Omega\subset\RR$ with finite measure, letting $B_{r(\alpha)}\subset\Omega^\sharp$ be the ball centered at the origin such
that $Q_p(\alpha,\Omega)= Q_p(\alpha,B_{r(\alpha)})$, we have
\begin{equation}\label{mainre}
\lambda_1(\Omega)\ge \lambda_1(B_{r(\alpha)}).
\end{equation}
\end{theorem}

\begin{remark} If in Theorem \ref{final} we put  $\alpha =0$, we obtain the inequality proved in \cite{KMCb}, when $p=2$, or in \cite{Br}, when $p>1$.
\end{remark}

Before giving the proof of Theorem \ref{final}, we prove the following result.
\begin{proposition}\label{monot}
For any $\alpha\in(-\infty,\lambda_1(\Omega))$, let $r(\alpha)>0$ be such that
$Q_p(\alpha,\Omega)= Q_p(\alpha,B_{r(\alpha)})$. The function $r(\alpha)$ is decreasing.
\end{proposition}

\begin{proof}
Using notation \eqref{Qrad}, $r(\alpha)$ is implicitly defined by the equality
\[Q_p^\sharp(\alpha,r)-Q_p(\alpha,\Omega)= 0.\]
Collecting Proposition \ref{propa} (\ref{(ab)}), Proposition \ref{prad}
and Proposition \ref{derQ}, and applying the implicit function theorem,
we get that $r(\alpha)$ is differentiable and 
\[
r'(\alpha)=-\frac{\frac\partial{\partial\alpha}(Q_p^\sharp(\alpha,r)-Q_p(\alpha,\Omega))}{\frac\partial{\partial r}Q_p^\sharp(\alpha,r)},
\]
where 
$\frac\partial{\partial r}Q_p^\sharp(\alpha,r)>0$ 
and
\[
\frac\partial{\partial \alpha}\bigl(Q_p^\sharp(\alpha,r)-Q_p(\alpha,\Omega)\bigr)=\int_{B_{r(\alpha)}} |\bar v(x)|^p\, \di x-
\int_\Omega |v(x)|^p\, \di x,
\]
where $v$ is the solution to \eqref{veq.0} in $\Omega$ and $\bar v$ is the solution to \eqref{veq} in $B_{r(\alpha)}$.
By Theorem \ref{comp} we have
$$
r'(\alpha)\le 0
$$
and the claim is proved.
\end{proof}

\begin{proof}[Proof of Theorem \ref{final}]
Let us observe that Proposition \ref{prad} implies that the ball
$B_{r(\alpha)}\subset\Omega^\sharp$ is well defined.
Since $\lambda_1(\Omega)>0$, in the limiting procedure
$\alpha\to\lambda_1(\Omega)^-$ we may assume $\alpha>0$.
For such values of $\alpha$, being $Q_p(\alpha,B_{r(\alpha)})$ finite,
Proposition \ref{prop} (\ref{(a)}) gives
\[
\alpha<\lambda_1(B_{r(\alpha)}).
\]
Equivalently,
\[
r(\alpha)<\bar r_\alpha,
\]
where $\bar r_\alpha>0$ is such that
\[
\lambda_1(B_{\bar r_\alpha})=\alpha.
\]

Hence, the monotonicity of $r(\alpha)$, proven in Proposition \ref{monot},
implies that the following limit exists
\[
\lim_{\alpha\rightarrow\lambda_1(\Omega)^{-}}r(\alpha)=\ell.
\]
We claim that $\ell\ge\tilde r$, where $\tilde r>0$ is such that
\[
\lambda_1(B_{\tilde r})=\lambda_1(\Omega).
\]
If, by contradiction,
\[
\ell<\tilde r ,
\]
then, taking into account that
\[
\lambda_1(B_{\ell})>\lambda_1(B_{\tilde r})=\lambda_1(\Omega),
\]
we would get
\[
\lim_{\alpha\rightarrow \lambda_1(\Omega)^{-}} Q_p(\alpha,\Omega)
=
\lim_{\alpha\rightarrow \lambda_1(\Omega)^{-}}
Q_p^\sharp(\alpha,r(\alpha))
=
Q_p^\sharp(\lambda_1(\Omega),\ell)<+\infty,
\]
in contrast with Proposition \ref{propa} (\ref{(ad)}). Therefore
\[
\lim_{\alpha\rightarrow\lambda_1(\Omega)^{-}}r(\alpha)\ge\tilde r.
\]
Since $r(\alpha)$ is decreasing, this implies
\[
r(\alpha)\ge\tilde r
\qquad\text{for every } \alpha<\lambda_1(\Omega).
\]
Finally, by the monotonicity of the first eigenvalue with respect to
set inclusion, we obtain
\[
\lambda_1(\Omega)=\lambda_1(B_{\tilde r})
\ge
\lambda_1(B_{r(\alpha)}).
\]
\end{proof}

\section*{Acknowledgement}

\noindent The research of the  second and third author was partially supported by Italian MIUR  through research projects   PRIN 2022: PRIN20229M52AS Partial differential equations and related geometric-functional inequalities and PRIN PNRR 2022 - P2022YFAJH - Linear and Nonlinear PDE's: New directions and Applications. The research of the  first author was partially supported by Italian MIUR  through research projects  PRIN PNRR 2022 - P2022YFAJH - Linear and Nonlinear PDE's: New directions and Applications.  The first, second and third authors are members of the Gruppo Nazionale per l'Analisi Matematica, la Probabilit\`a
e le loro Applicazioni (GNAMPA) of the Istituto Nazionale di Alta Matematica (INdAM).
The research of the last author was supported in part by NSF Grant DMS-2246817.


\bigskip

\begin{thebibliography}{20}


\bibitem{AFT}  {\sc A. Alvino, V. Ferone, G. Trombetti}, \emph{On the properties of some nonlinear eigenvalues}, SIAM J. Math. Anal. 29 (1998), 437--451.

\bibitem{ALT} {\sc A. Alvino, P.-L. Lions, G. Trombetti}, \emph{On optimization problems with prescribed rearrangements}, Nonlinear Anal. Theory Methods Appl. 13 (1989), 185--220.


\bibitem{Ba} {\sc C. Bandle}, \emph{Bounds for the solutions to boundary value problems}, J. Math. Anal. Appl.,  54 (1976), 706--716.

\bibitem{Bandle} {\sc C. Bandle},  Isoperimetric inequalities and applications, Monographs and Studies
in Mathematics, vol. 7, Pitman (Advanced Publishing Program), Boston, MA 1980.

\bibitem{BS} {\sc C. Bennett, R. Sharpley},  Interpolation of operators,  Pure and Applied Mathematics, vol. 129, Academic Press Inc., Boston, MA 1988.

 
%
\bibitem{BoccardoMurat} {\sc L. Boccardo, F. Murat}, \emph{Almost everywhere convergence of the gradients of solutions to elliptic and parabolic equations}, Nonlinear Anal. 19 (1992), 581--597.
%
%

\bibitem{Br}
{\sc L.~Brasco}, {\em On torsional rigidity and principal frequencies: an invitation to the {K}ohler-{J}obin rearrangement technique}, ESAIM: Control, Optimisation and Calculus of Variations, 20 (2014), 315--338.

\bibitem{BO} {\sc H. Brezis, L. Oswald}, \emph{Remarks on sublinear elliptic equations},
 Nonlinear Anal.  10  (1986),  55--64.

 \bibitem{BZ}
{\sc J. E. Brothers, W. P. Ziemer},
\emph{ Minimal rearrangements of Sobolev functions},
 Reine Angew. Math. 384 (1988), 153--179.

\bibitem{C1} {\sc G. Chiti}, \emph{An isoperimetric inequality for the eigenfunctions of linear second order elliptic operators}, Boll. Un. Mat. Ital. A(6) 1 (1982), 145--151. 
 
\bibitem{C2} {\sc G. Chiti}, \emph{A reverse H\"older inequality for the eigenfunctions of linear second order elliptic operators}, Z. Angew. Math. Phys. 33 (1982), 143--148. 

\bibitem{ChongRice} {\sc K.M. Chong,  N.M. Rice},
\emph{ Equimeasurable rearrangements of functions},
Queen's Papers in Pure and Applied Mathematics, No. 28. Queen's University, Kingston, ON, 1971. 

\bibitem{DS} {\sc J.I. D\'\i az, J.E. Sa\'a}, \emph{Existence et unicit\'e de solutions positives pour certaines \'equations elliptiques quasilin\'eaires}, C. R. Acad. Sci. Paris S\'er. I Math.  305  (1987),  521--524.

%
%

\bibitem{Fa}{\sc G. Faber}, \emph{Beweis, da{{\ss}} unter allen homogenen {Membranen} von gleicher {Fl{\"a}che} und gleicher {Spannung} die kreisf{\"o}rmige den tiefsten {Grundton} gibt}, Sitzungsber., Bayer. Akad. Wiss., Math.-Naturwiss. Kl., (1923), 169--172.

\bibitem{GMS} {\sc J. Garc\'{\i}a Meli\'{a}n, J. Sabina de Lis}, \emph{On the perturbation of eigenvalues for the {$p$}-{L}aplacian}, C. R. Acad. Sci. Paris S\'{e}r. I Math. 332 (2001),
893--898.
     
\bibitem{HL} {\sc O. Herscovici, G.V. Livshyts}, \emph{Kohler-Jobin meets Ehrhard: the sharp lower bound for the Gaussian principal frequency while the Gaussian torsional rigidity is fixed, via rearrangements},
 Proc. Amer. Math. Soc.  152  (2024), 4437--4450.
 
\bibitem{KMCb} {\sc M.-T. Kohler-Jobin},
     \emph{D\'{e}monstration de l'in\'{e}galit\'{e} isop\'{e}rim\'{e}trique {$P\lambda^{2}\geq \pi j^{4}_{0}/2$}, con\-jectur\'{e}e par {P}\'{o}lya et
              {S}zeg\H{o}},
{C. R. Acad. Sci. Paris S\'{e}r. A-B} {281} (1975),
A119--A121.
		
\bibitem{KMCa} {\sc M.-T. Kohler-Jobin},
\emph{Une propri\'{e}t\'{e} de monotonie isop\'{e}rim\'{e}trique qui contient plusieurs th\'{e}or\`emes classiques},
 {C. R. Acad. Sci. Paris S\'{e}r. A-B} {284} ({1977}),
{A917--A920}.

\bibitem{K} {\sc M.-T. Kohler-Jobin}, \emph{Sur la premi\'ere fonction propre d'une membrane: une extension \`a $N$ dimensions de l'in\'egalit\'e isop\'erim\'etrique de Payne-Rayner}, Z. Angew. Math. Phys. 28 (1977), 1137--1140.

\bibitem{KMZa} {\sc M.-T. Kohler-Jobin},
\emph{Une m\'{e}thode de comparaison isop\'{e}rim\'{e}trique de fonctionnelles
de domaines de la physique math\'{e}matique. {I}. {U}ne
d\'{e}monstration de la conjecture iso\-p\'{e}rim\'{e}trique {$P\lambda^{2}\geq \pi j^{4}_{0}/2$} de {P}\'{o}lya et {S}zeg\H{o}},
 {Z. Angew. Math. Phys.} {29} (1978),
 {757--766}.
				
\bibitem{KMZb} {\sc M.-T. Kohler-Jobin},
\emph{Une m\'{e}thode de comparaison isop\'{e}rim\'{e}trique de fonctionnelles
de domaines de la physique math\'{e}matique. {II}. {C}as inhomog\`ene: une in\'{e}galit\'{e} isop\'{e}rim\'{e}trique entre la fr\'{e}quence
fondamentale d'une membrane et l'\'{e}nergie d'\'{e}quilibre d'un probl\`eme de {P}oisson},
 {Z. Angew. Math. Phys.} {29} (1978), 767--776.
		
\bibitem{KMZc} {\sc M.-T. Kohler-Jobin},
\emph{Isoperimetric monotonicity and isoperimetric inequalities of Payne-Rayner type for the first eigenfunction of the Helmholtz problem},
 {Z. Angew. Math. Phys.} {32} (1981), 625--646.

\bibitem{Kr}
{\sc E. Krahn}, 
\emph{\"Uber eine von Rayleigh formulierte Minimaleigenschaft des Kreises}, Math. Ann. 94 (1924), 97--100.



\bibitem{Lam}
{\sc P. D. Lamberti},
\emph{A differentiability result for the first eigenvalue of the $p$-Laplacian upon domain perturbation}, 
 Nonlinear Analysis and Applications, Kluwer Academic Publishers, Dordrecht, 2003, pp. 741--754.


\bibitem{Lin} {\sc P. Lindqvist},
\emph{On the equation {${\rm div}\,(|\nabla u|^{p-2}\nabla
u)+\lambda|u|^{p-2}u=0$}}, {Proc. Amer. Math. Soc.} 109 (1990), 157--164.

%
%
%
\bibitem{Po}{\sc G. P{\'o}lya},
\emph{Torsional rigidity, principal frequency, electrostatic capacity and symmetrization},
Q. Appl. Math., {6} (1948),
{267--277}.

\bibitem{PS}{\sc G. P{\'o}lya, G. Szeg{\H o}},
\textit{Isoperimetric inequalities in mathematical physics},
{Ann. Math. Stud.},
vol. {27},
Princeton University Press, Princeton, NJ, {1951}.

\bibitem{T} {\sc G. Talenti}, 
\emph{Elliptic equations and rearrangements}, Ann. Scuola Norm. Sup. Pisa Cl. Sci. (4) 3 (1976), 697--718.

\bibitem{T1} {\sc G. Talenti}, 
\emph{Nonlinear elliptic equations, rearrangements of functions and {O}rlicz spaces},
Ann. Mat. Pura Appl. (4) {120} (1979), 160--184.

\bibitem{T2} {\sc G. Talenti},
\emph{A weighted version of a rearrangement inequality},
Ann. Univ. Ferrara Sez. VII (N.S.) 43 (1997), 121--133.

\end{thebibliography}
\end{document}